\documentclass[12pt]{amsart}

\usepackage[utf8]{inputenc}
\usepackage{geometry,amssymb,amsthm,amsmath,graphics,enumerate}
\usepackage{enumerate}
\usepackage{url}
\usepackage{hyperref}

\makeatletter
\@namedef{subjclassname@2010}{%
  \textup{2010} Mathematics Subject Classification}
\makeatother

\theoremstyle{plain}
\newtheorem{Thm}{Theorem}[section]
\newtheorem{Lem}[Thm]{Lemma}

\newtheorem{Cor}[Thm]{Corollary}

\newtheorem{Def}[Thm]{Definition}
\newtheorem{Rem}[Thm]{Remark}

\newcommand{\lomegaone}{\ensuremath{\mathcal{L}_{\omega_1,\omega}}}

\newcommand{\fA}{\ensuremath{\mathfrak{A}}}
\newcommand{\fB}{\ensuremath{\mathfrak{B}}}
\newcommand{\fM}{\ensuremath{\mathfrak{M}}}

\newcommand{\restrict}{\upharpoonright}

\newcommand{\fried}{FriedmanEtAllNonAbsolutenessOfModelExistence}
\newcommand{\suchthat}{\mathrel{\mid}}
\newcommand{\inv}[1]{{#1}^{-1}}
\newcommand{\card}[1]{\left\lvert #1\right\rvert}

\newcommand{\alp}[1]{\aleph_{#1+1}}
\newcommand{\at}[1]{$#1$-Aronszajn tree}

\newcommand{\sat}[1]{special \at{#1}}
\newcommand{\csat}[1]{coherent \sat{#1}}
\newcommand{\alephs}[1]{\ensuremath{\aleph_{#1}}}
\def\bk{\mbox{$\mathbf{K}$}}

\def\bkalpha{\mbox{$\mathbf{K}_\alpha$}}

\DeclareMathOperator{\cf}{cf}
\DeclareMathOperator{\ran}{ran}
\DeclareMathOperator{\dom}{dom}

\title{Non-Absoluteness of Model Existence at $\aleph_\omega$}
\author{David Milovich}
\address{Dept. of Mathematics and Physics, Texas A\&M International University, 5201 University Blvd., Laredo, TX 
78045, 
USA}
\email{david.milovich@tamiu.edu}

\author{Ioannis Souldatos }
\address{4001 W.McNichols, Mathematics Department, University of Detroit Mercy, Detroit, MI 48221, USA}
\email{souldaio@udmercy.edu}
\date{}



\begin{document}

\baselineskip=17pt

\begin{abstract}
In \cite{FriedmanEtAllNonAbsolutenessOfModelExistence},
the authors considered the question whether model-existence of 
$\lomegaone$-sentences is absolute for transitive models of ZFC,
in the sense that  if $V \subseteq W$ are transitive 
models of ZFC with the same ordinals, $\varphi\in V$ and
$V\models ``\varphi \text{ is an } \lomegaone\text{-sentence}"$,
then $V \models\Phi$ if and only if $W \models\Phi$ where $\Phi$
is a first-order sentence with parameters $\varphi$ and $\alpha$
asserting that $\varphi$ has a model of size $\aleph_\alpha$.

From \cite{FriedmanEtAllNonAbsolutenessOfModelExistence} we know that the answer is positive for $\alpha=0,1$ and under 
the negation of CH,
the answer is negative for all $\alpha>1$.
Under GCH, and assuming the consistency of a supercompact cardinal,
the answer remains negative for each $\alpha>1$, except the 
case when $\alpha=\omega$ which is an open question in \cite{FriedmanEtAllNonAbsolutenessOfModelExistence}.

We answer the open question by providing a negative answer under GCH even for $\alpha=\omega$. Our examples are incomplete 
sentences. In fact, the same sentences can be 
used to prove a negative answer under GCH for all $\alpha>1$ assuming the consistency of a Mahlo cardinal. Thus, the large 
cardinal assumption is relaxed from a supercompact in \cite{FriedmanEtAllNonAbsolutenessOfModelExistence} to a Mahlo cardinal.

Finally, we consider the absoluteness question for the 
$\aleph_\alpha$-amalgamation property of $\lomegaone$-sentences (under \emph{substructure}). We prove that assuming GCH, 
$\aleph_\alpha$-amalgamation is non-absolute for $1<\alpha<\omega$. This answers a question from 
\cite{SinapovaSKurepa}. The cases $\alpha=1$ and $\alpha$ infinite remain open. As a corollary we get that it is non-absolute 
that the amalgamation spectrum of an $\lomegaone$-sentence is empty.
\end{abstract}
\maketitle

\section{Introduction}

The current paper adds to the literature that investigates which notions for infinitary logics, or more generally for 
abstract elementary classes,  are absolute for models of ZFC. Some notions like satisfiability\footnote{We mean 
that given a model $M$ and a sentence $\phi$, the statement ``$M\models\phi$'' is absolute.}, 
model-existence in $\aleph_0$ and $\aleph_1$, model-existence in some $\kappa\ge\beth_{\omega_1}$, $\aleph_0$-amalgamation and 
$\aleph_0$-joint embedding are absolute between transitive models of ZFC
(see \cite{BaldwinAmalgamation,FriedmanEtAllNonAbsolutenessOfModelExistence, GrossbergShelahNonisomorphic}). 
Other notions such as model-existence in $\aleph_\alpha$, $\alpha>1$, or existence of a maximum model in 
$\aleph_\alpha$, $\alpha>1$ are non-absolute (see 
\cite{FriedmanEtAllNonAbsolutenessOfModelExistence, JEP, Complete}). Unfortunately, the absoluteness question remains open for 
a wide range of notions, such as $\aleph_1$-categoricity for $\lomegaone$-sentences, 
$\aleph_1$- amalgamation, and  $\aleph_1$-joint embedding, to name a few.

The notions we consider in this paper are ``model-existence'' and ``amalgamation''. For $\aleph_0$ and $\aleph_1$, model 
existence is an absolute 
notion for transitive models of ZFC. From \cite{MalitzsHanfNumber}, we know that there is a complete $\lomegaone$-sentence 
$\phi$ that characterizes $2^{\aleph_0}$. That is, $\phi$ has models in all (infinite) cardinalities less or equal to 
$2^{\aleph_0}$, but no larger models, and this is a theorem of ZFC. Under CH, $\phi$ has models in $\aleph_1$, but no models 
in $\aleph_2$. Under the negation of CH, $\phi$ has a model in $\aleph_2$. Hence, model-existence in $\aleph_2$ is not an 
absolute notion. 

Similarly, other consistent violations of GCH witness that
for each $1<\alpha<\omega_1$ model existence in $\aleph_\alpha$ is not absolute.

Recall that $\beth_{\omega_1}$ is the Hanf number for $\lomegaone$. {\it I.e.}, every $\lomegaone$-sentence which has models in all 
cardinalities below $\beth_{\omega_1}$, it also has arbitrarily large models.  By 
\cite{GrossbergShelahNonisomorphic}, the property that an $\lomegaone$-sentence has arbitrarily large models is absolute.  

So, it is natural to ask whether model-existence in $\aleph_\alpha$, with $\aleph_1<\aleph_\alpha<\beth_{\omega_1}$, is 
absolute for models of ZFC+GCH. The question was answered in \cite{FriedmanEtAllNonAbsolutenessOfModelExistence}, under large 
cardinal assumptions, except the case where $\alpha=\omega$. The large cardinal assumptions are different for successors of 
successors than for limit cardinals and successors of limits. 

The following result from \cite{FriedmanEtAllNonAbsolutenessOfModelExistence},
shows that, assuming the consistency of uncountably many inaccessibles,
model-existence in $\aleph_{\alpha+2}$ is a non-absolute notion
for $\lomegaone$-sentences, for all $\alpha<\omega_1$.

\begin{Thm}[\cite{FriedmanEtAllNonAbsolutenessOfModelExistence}, Theorem 7]
  \label{FKT1} Assume a ground model $V$ of ZFC+GCH in which
  there are uncountably many inaccessible cardinals
  and an inner model $M\subset V$ of ZFC+GCH and
  ``$\,\Diamond^+_\kappa$ holds for every regular uncountable $\kappa<\aleph_{\omega_1}$.''
  Then there is a generic extension $V[G]$ in which the GCH is true and
  model-existence in $\aleph_{\alpha+2}$ for $\lomegaone$-sentences
  is not absolute between $M$ and $V[G]$ for all $\alpha<\omega_1^M$.\footnote{
      The proofs of Theorems~\ref{FKT1}-\ref{FKT3}
      obtain $\omega_1^{V[G]}=\omega_1^V$.
  }
\end{Thm}

The way the proof of Theorem \ref{FKT1} goes is that for each $\alpha<\omega_1^M$ there exists an 
$\lomegaone$-sentence $\sigma^{\alpha+2}$ such that $\sigma^{\alpha+2}$ has a model of size $\alephs{\alpha+2}$ if and only if 
there exists an $\alephs{\alpha+1}$-Kurepa family.
Moreover, this equivalence is absolute between transitive models of ZFC that contain $\sigma^{\alpha+2}$.
It follows that $\sigma^{\alpha+2}$ has a model of size $\alephs{\alpha+2}$ in $M$.
In addition, since Levy collapsing an inaccessible to $\aleph_{\alpha+2}$
destroys all $\aleph_{\alpha+1}$-Kurepa trees, there is a generic extension of $V$ where $\sigma^{\alpha+2}$ does not have any models of 
size $\alephs{\alpha+2}$.
So, another form of Theorem~\ref{FKT1} is implicit in~\cite{\fried}.
Assuming $M$ and $V$ as in Theorem~\ref{FKT1} except now with just one inaccessible cardinal,
then, for each $\alpha<\omega_1^M$,
model-existence in $\aleph_{\alpha+2}$ for $\lomegaone$-sentences
is not absolute between $M$ and a forcing extension $V[G_\alpha]$ satisfying GCH.
This covers the case of successors of successor cardinals.

The following two results from \cite{FriedmanEtAllNonAbsolutenessOfModelExistence},
show that, assuming the consistency of a supercompact,
model-existence in $\aleph_\beta$ for $\lomegaone$-sentences
is not absolute between transitive models of ZFC+GCH,
for every countable $\beta>\omega$ not of the form $\alpha+2$.

\begin{Thm}[\cite{FriedmanEtAllNonAbsolutenessOfModelExistence}, Theorem 6]
  \label{FKT2} Assume a ground model $V$ of ZFC+GCH in which
  there is a supercompact cardinal
  and an inner model $M\subset V$ of ZFC+GCH and
  ``$\,\square^*_\lambda$ holds at every singular cardinal $\lambda<\aleph_{\omega_1}$.''
  Then there is a generic extension $V[G]$ in which the GCH is true and
  model-existence in $\aleph_{\alpha+1}$ for $\lomegaone$-sentences
  is not absolute between $M$ and $V[G]$ for all limit $\alpha<\omega_1^M$.
\end{Thm}

\begin{Thm}[\cite{FriedmanEtAllNonAbsolutenessOfModelExistence}, Section 3.5]
  \label{FKT3} Assume $M$ and $V$ is as in Theorem~\ref{FKT2}.
  Then there is a generic extension $V[G]$ in which the GCH is true and
  model-existence in $\aleph_\alpha$ for $\lomegaone$-sentences
  is not absolute between $M$ and $V[G]$
  for all limit $\alpha<\omega_1^M$ except possibly $\omega$.
\end{Thm}

Theorem \ref{FKT2} is proved as follows:
Given $\alpha<\omega_1^M$, there is an $\lomegaone$-sentence $\varphi^{\alpha+1}$
which has a model of size $\alephs{\alpha+1}$ if and only if there is a \sat{\alp{\alpha}}.
Moreover, this equivalence is absolute between
transitive models of ZFC that contain $\varphi^{\alpha+1}$.
Therefore, $\varphi^{\alpha+1}$ has a model of size $\alephs{\alpha+1}$ in $M$.
Moreover, assuming a supercompact, there is a forcing extension $V[G]$
in which GCH is true, but there are no \sat{\alp{\alpha}}s, for all countable limit $\alpha$. 

Theorem \ref{FKT3} has a similar proof, but now for every limit $\omega<\alpha<\omega_1^M$, there exists some $\lomegaone$-sentence 
$\psi^\alpha$ that codes multiple special Aronszajn trees simultaneously.
The vocabulary of $\psi^\alpha$ contains predicates  $Q_\beta$, for all $\beta<\alpha$. Each 
$Q_{\beta+1}$ is $|Q_\beta|$-like and if $|Q_{\beta+1}|$=$|Q_\beta|^+$, then $\psi^\alpha$ codes a special 
$|Q_{\beta+1}|$-Aronszajn tree. 
It follows that $\psi^\alpha$ has models of size $\aleph_\alpha$
if and only if there are \sat{\alp{\beta}}s for all $\beta<\alpha$.
Moreover, this equivalence is absolute between
transitive models of ZFC that contain $\psi^\alpha$.
Thus, $\psi^\alpha$ has models of size $\aleph_\alpha$ in $M$
while in the generic extension $V[G]$ used in the proof of Theorem \ref{FKT2},
all models of $\psi^\alpha$ have size at most $\aleph_\omega$.

The above argument fails for $\alpha=\omega$ because
GCH implies a \sat{\alp{n}} for each $n<\omega$.
We overcome this barrier by using \emph{coherent} special $\kappa$-Aronszajn trees (see Definition 
\ref{pseudotree} for coherence).

The following is Corollary 3.15(a) in  \cite{KonigCoherence}. 
\begin{Thm}\label{Konig} If $\square_\lambda$ holds, then there is a \csat{\lambda^+}.  
\end{Thm}

So, in $L$, our modified formula $\phi_\alpha$
will have a model of size $\aleph_\alpha$.
On the other hand,
from a model of $\phi_\alpha$ of size $\geq\aleph_2$,
we can recover a coherent pseudotree
that contains a cofinal \sat{\aleph_2}. By coherence, this pseudotree cannot contain a copy of $2^{\leq\omega}$.
Todor\v cevi\'c showed that after Levy collapsing a Mahlo to $\omega_2$,
every \sat{\aleph_2} contains a copy of $2^{<\omega_1}$.
(We state his corresponding equiconsistency theorem below.)
Therefore, there is a model of GCH in which $\phi_\alpha$ has no models 
of size $\geq\aleph_2$.

\begin{Thm}[\cite{TodorcevicTrees}, Theorem 4.6]\label{Todorcevic} Con(ZFC+``there exists a Mahlo'')$\leftrightarrow$ 
Con(ZFC+GCH+``every special $\aleph_2$-Aronszajn tree contains a copy of $2^{<\omega_1}$'').
\end{Thm}

Our result works not only for $\alpha=\omega$, but for all countable $\alpha$. For $\alpha$ a limit ordinal, or the 
successor to a limit ordinal, our result improves the large cardinal assumption from a supercompact (Theorems 
\ref{FKT2}, \ref{FKT3}) to a Mahlo cardinal (Theorem \ref{maintheorem}).

This completes all cases of the absoluteness question for  model-existence of $\lomegaone$-sentences. The same question can be 
asked about the amalgamation property of $\lomegaone$-sentences. Before we phrase the question precisely, notice that for the 
amalgamation property we need to specify the type of embeddings used. For this paper we consider only amalgamation under the 
substructure relation.  

\begin{Def}\label{apspec} Given a collection of models $\mathbf{K}$,
  by the \emph{amalgamation spectrum} of $\mathbf{K}$, in symbols AP-spec$(\bk)$,
  we mean the set of all cardinals $\kappa$ for which
  the class of all models in $\mathbf{K}$ of size $\kappa$
  is nonempty and has the amalgamation property.
  
  If $\bk$ is the collection of all models of some sentence $\varphi$, then we write AP-spec($\varphi)$ for AP-spec$(\bk)$.
\end{Def}

Then the absoluteness question for amalgamation is the following:
Is it the case that, if $V \subseteq W$ are 
transitive models of ZFC with the same ordinals, $\varphi\in V$ and
$V\models ``\varphi \text{ is an } \lomegaone\text{-sentence}$",
then $V \models ``\aleph_\alpha\in \text{AP-spec}(\varphi)"$
if and only if
$W \models ``\aleph_\alpha\in \text{AP-spec}(\varphi)"$?

Parallel to \cite{FriedmanEtAllNonAbsolutenessOfModelExistence}, we can show that manipulating the size of the continuum yields a 
non-absoluteness result for the amalgamation spectrum of $\lomegaone$-sentences.

Consider the sentence $\phi$ that asserts the 
existence of a full binary tree of length $\omega$. This sentence has models up to cardinality continuum. All models of $\phi$ 
differ only on the maximal branches they contain. In particular, they satisfy the amalgamation property in all cardinals up to 
the 
continuum. It follows that the $\kappa$-amalgamation property is not 
absolute for $\kappa\ge\aleph_2$. A similar result, but for $\aleph_2\le\kappa\le 2^{\aleph_1}$, is 
proved in \cite{SinapovaSKurepa} using Kurepa trees. The result is interesting mainly when GCH fails, since under GCH, 
$\aleph_2=2^{\aleph_1}$. In \cite{SinapovaSKurepa}, the question was raised about the absoluteness of $\kappa$-amalgamation, for
$\kappa\ge\aleph_3$, assuming GCH. In Section \ref{Sec:AP} we answer the question for all $\kappa=\aleph_\alpha$, $3\le 
\alpha<\omega$, and we prove that our examples cannot be used to settle the question for $\alpha\ge\omega$.

\section{Model- Existence}
In this section we use coherent special Aronszajn trees to prove Theorem \ref{maintheorem} about non-absoluteness of 
model-existence. 
Recall that well-orderings cannot be characterized by an
$\lomegaone$-sentence. So, it is unavoidable that we will be working with non-well-founded trees. We call such trees
\emph{pseudotrees} to distinguish them from their well-founded counterparts. 

\begin{Def}\label{pseudotree}~
  \begin{itemize}
  \item A \emph{pseudotree} is a partial ordered set
    $T$ such that each strict lower cone
    $\downarrow x=\{y\suchthat y<_T x\}$ is a chain.
  \item A pseudotree $T$ is \emph{functional}
    if there is a linear order $L$
    such that $T$ is a downward closed suborder of
    the class of all functions with domains
    of the form $\downarrow x=\{y\suchthat y<_L x\}$,
    ordered by inclusion. In this case,
    define a \emph{rank} $\rho\colon T\to L$
    by $\rho(t)$ being the unique element such that $\dom(t)={\downarrow\rho(t)}$.
  \item Given $T$ and $L$ as above and $x\in L$,
    define $T_x$ to be the fiber $\inv{\rho}(x)$.    
  \item The \emph{cofinality} $\cf(T)$
    of a functional pseudotree $T$
    is the cofinality $\cf(\rho[T])$.
  \item By $=^*$ we mean equality of sets
    modulo a finite set.
  \item A pseudotree $T$ is \emph{coherent} if
    it is functional and $\dom(s)=\dom(t)$ implies $s=^*t$.
  \item Given a regular uncountable cardinal $\kappa$,
    a \emph{$\kappa$-pseudotree} is a pseudotree $T$
    of cofinality $\kappa$ such that
    $\card{T_x}<\kappa$ for each $x\in\rho[T]$.
  \item A $\kappa^+$-pseudotree is \emph{special} if 
    it is the union of $\kappa$-many of its antichains.
 \end{itemize}
\end{Def}

\begin{Lem}\label{nocantorsubtree} If $T$ is a coherent pseudotree $T$ of uncountable cofinality,
  no suborder of $T$ is isomorphic to $(2^{\leq\omega},\subset)$.
\end{Lem}
\begin{proof}
  Seeking a contradiction,
  suppose $e\colon2^{\leq\omega}\to T$ is an order embedding.
  Choose $t\in T$ such that $\rho(t)\geq\rho(e(c))$ for all $c\in 2^{<\omega}$.
  This is possible since $T$ has uncountable cofinality. 
  Then construct $w\in 2^\omega$ as follows.
  Given $c=w\restrict n$,
  since $e(c^\frown 0)\perp e(c^\frown 1)$,
  we may choose $w(n)=i<2$ such that
  $e(c^\frown i)(y)\not=t(y)$ for some $y\geq\rho(e(c))$. 
  Thus, $e(w)(y)\not=t(y)$ for infinitely many $y$,
  in contradiction with coherence of $T$.
\end{proof}

\begin{Lem}\label{mainlemma}
  Given $1\leq\alpha<\omega_1$, there is an $\lomegaone$
  formula $\phi_\alpha$ satisfying the following.
  \begin{enumerate}
  \item If $\phi_\alpha$ has a model $\fA$ of size $\geq\aleph_2$,
    then there is a coherent pseudotree $T$
    with cofinality $\omega_2$ and an
    order embedding of a \sat{\aleph_2}\ into $T$.
  \item If there is a \csat{\alp{\beta}} for each $\beta<\alpha$,
    then $\phi_\alpha$ has a model $\fB$ of size $\aleph_\alpha$.
    \item There is no model of $\phi_\alpha$ of size greater than $\aleph_\alpha$. 
  \end{enumerate}
\end{Lem}
\begin{proof}
  Let $1\leq\alpha<\omega_1$.
  We will use a predicate symbol $\omega_\beta$ for each $\beta\le\alpha$,
  a binary relation symbol $<$,
  ternary relation symbols $L_\beta$ and $S_\beta$ for each $\beta<\alpha$,
  and a $4$-ary relation symbol $T_\beta$ for each $\beta<\alpha$.
  Our sentence $\phi_\alpha$ will assert that the predicates
  $L_\beta$, $T_\beta$, and $S_\beta$ are functional, {\it i.e.},
  each of these predicates defines the graph of a function.
  Therefore, we will freely write, for example,
  $z=L_\beta(x,y)$ to denote the unique $z$ such that $L_\beta(x,y,z)$.
  Further, $L_\beta(x,\bullet)$ will denote the function
  sending $y$ to $L_\beta(x,y)$.
  
  The idea behind the definition below is that for each $\beta<\alpha$,
  the relation $T_\beta$ defines a functional pseudotree
  with underlying order $\omega_{\beta+1}$.
  $T_\beta(x,\bullet,\bullet)$ will enumerate
  the set of all functions in the pseudotree with rank equal to $x$.
  Each such function will equal $T_\beta(x,y,\bullet)$,
  for some $y\in \omega_\beta$.
  So, each level of the pseudotree will have size at most $|\omega_\beta|$.
  Also, $\dom(T_\beta(x,y,\bullet))={\downarrow x}$.
  We will use $S_\beta$ to witness the fact that $T_\beta$ is special. 

  Let $\phi_\alpha\in\lomegaone$ assert the following statements
  for each $\beta<\alpha$, $x\in\omega_{\beta+1}$, and $y\in\omega_\beta$.
  
  \begin{enumerate}
  \item The universe is a continuously increasing union
    $\bigcup_{\beta\le\alpha}\omega_\beta$
    and strictly linearly ordered by $<$.
  \item $\omega_0$ is infinite yet
    $\downarrow n=\{m\suchthat m<n\}$ is finite
    for all $n\in\omega_0$.
  \item $\omega_\beta$ is $<$-downward closed.
  \item $L_\beta$, $T_\beta$ and $S_\beta$ are functional predicates.
    
  \item\label{omegabetalike} If $\downarrow x$ is not empty,
    then $L_\beta(x,\bullet)$ is a surjection from
    $\omega_\beta$ to $\downarrow x$.
    This will ensure that each $\omega_{\beta+1}$ is $|\omega_\beta|$-like.
  \item $\dom(T_\beta(x,y,\bullet))={\downarrow x}$ and
    $T_\beta(x,y,\bullet)=^*T_\beta(x,z,\bullet)$, for each 
    $z\in\omega_\beta$, as required for coherence.
  \item For each $w<x$, there exists $z\in\omega_\beta$
    such that $T_\beta(w,z,\bullet)\subset T_\beta(x,y,\bullet)$,
    {\it i.e.}, $T_\beta$ is downward closed.
  \item $S_\beta: \omega_{\beta+1}\times\omega_\beta\to\omega_\beta$. 
  \item For each $w<x$ and $z\in\omega_\beta$,
    if $T_\beta(w,z,\bullet)\subset T_\beta(x,y,\bullet)$,
    then $S_\beta(w,z)\not=S_\beta(x,y)$.
  \end{enumerate}

  Assuming there exists a model $\fA$ of $\phi_\alpha$
  of size $\geq\aleph_2$,
  there is a unique $\beta<\alpha$ such that
  $\card{\omega_\beta^{\fA}}=\aleph_1$ and
  $\cf(\omega_{\beta+1}^{\fA})=\omega_2$.
  Why? First, by regularity of $\aleph_2$,
  the least $\gamma\leq\alpha$
  such that $\card{\omega_\gamma^{\fA}}\geq\aleph_2$
  must be a successor ordinal $\beta+1$.
  Second, by~\eqref{omegabetalike},
  $\card{I}\leq\card{\omega_\beta^{\fA}}$ 
  for every proper initial segment $I$ of $\omega_{\beta+1}^{\fA}$.
  Therefore, $\card{\omega_\beta^{\fA}}=\aleph_1$,
  $\card{\omega_{\beta+1}^{\fA}}=\aleph_2$,
  and $\omega_{\beta+1}^{\fA}$ cannot be covered
  by $\aleph_1$-many proper initial segments.
  
  Select $W\subset\omega_{\beta+1}^{\fA}$
  such that $(W,<)\cong (\omega_2,\in)$, 
  and define trees $T$ and $U$ as follows:
  \begin{align*}
    T&=\left(\left\{T^{\fA}_\beta(x,y,\bullet)\suchthat
    (x,y)\in \omega_{\beta+1}^{\fA}\times\omega_\beta^{\fA}\right\}
    ,\subset\right)\\
    U&=\left(\left\{T^{\fA}_\beta(x,y,\bullet)\suchthat
    (x,y)\in W\times\omega_\beta^{\fA}\right\},\subset\right)
  \end{align*}
  Then $T$ is a coherent pseudotree of cofinality $\omega_2$,
  $U$ is an $\omega_2$-tree and suborder of $T$,
  and $S_\beta^{\fA}$ witnesses that $U$ is a \sat{\aleph_2}.
  This proves part (1).
  
  For part (2), assuming the existence of
  a \csat{\alp{\beta}} $\Upsilon^{(\beta)}$
  for each $\beta<\alpha$, let us construct a model
  $\fB$ of $\phi_\alpha$ with size $\aleph_\alpha$.
  Without loss of generality, each $\Upsilon^{(\beta)}$
  is a downward closed suborder of
  $((\omega_{\beta+1})^{<\omega_{\beta+1}},\subset)$.
  For each $\beta<\alpha$, let
  $\Xi_\beta\colon\Upsilon^{(\beta)}\to\omega_\beta$ witness specialness.
  Let $\fB$ have universe $\omega_\alpha$
  with ${<^{\fB}}={\in}$ and $\omega_\beta^{\fB}=\omega_\beta$.
  For each $\beta<\alpha$ and $\gamma<\omega_{\beta+1}$:
  \begin{enumerate}
  \item If $\gamma\not=0$, choose a surjection
    $L_\beta(\gamma,\bullet)\colon\omega_\beta\to\gamma$ and let 
    $L_\beta^{\fB}(\gamma,\bullet)=L_\beta(\gamma,\bullet)$.
  \item Choose a surjection
    \(\Lambda_\beta(\gamma,\bullet)\colon\omega_\beta
    \to\Upsilon^{(\beta)}_\gamma\), where $\Upsilon^{(\beta)}_\gamma$
    is the $\gamma^{th}$ level of $\Upsilon^{(\beta)}$, and 
    let \(T^{\fB}_\beta(\gamma,\delta,\bullet)
    =\Lambda_\beta(\gamma,\delta)\) for each $\delta<\omega_\beta$.
  \item Let \(S^{\fB}_\beta(\gamma,\bullet)
    =\Xi_\beta(\Lambda_\beta(\gamma,\bullet))\).
  \end{enumerate}
  It is immediate that $\fB$ is a model of $\phi_\alpha$
  and $\fB$ has size $\aleph_\alpha$, which proves (2). We finish the proof by noticing that (3) follows directly from the 
definition. 
\end{proof}

Notice that $\phi_\alpha$ is an incomplete sentence.
\begin{Thm}\label{maintheorem}
  For each $2\le \alpha<\omega_1$, let $\phi_\alpha$ be the sentence from Lemma \ref{mainlemma}. 
  \begin{enumerate}
  \item Given $2\le \alpha<\omega_1$, if $\square_{\aleph_\beta}$ holds for all $\beta<\alpha$, then $\phi_\alpha$ has a model of size $\aleph_\alpha$.
    In particular, it is consistent with ZFC+GCH that,
    for all $2\le \alpha<\omega_1$,
    $\phi_\alpha$ has a model of size $\aleph_\alpha$.
\item It is consistent, relative to the existence of a Mahlo cardinal,
  that there is a model of ZFC+GCH in which,
  for each $2\le \alpha<\omega_1$,
  all models of $\phi_\alpha$ have size at most $\aleph_1$.
  \end{enumerate}
\end{Thm}
\begin{proof} (1) follows by Lemma \ref{mainlemma}, part (2), and Theorem \ref{Konig}. (2) follows from Lemma \ref{mainlemma}, 
part (1), Lemma \ref{nocantorsubtree} and Theorem \ref{Todorcevic}. 
\end{proof}

\section{Amalgamation}\label{Sec:AP}
In this section we consider the absoluteness question for the amalgamation spectra of $\lomegaone$-sentences. In particular, we 
investigate the amalgamation spectra of the sentence $\phi_\alpha$ from Lemma 
\ref{mainlemma} under the \emph{substructure} relation. We fix some notation first.

\begin{Def} For each $1\le \alpha<\omega_1$, let $(\bkalpha,\subset)$ be the collection of all models of $\phi_\alpha$ from 
Theorem \ref{maintheorem} equipped with the substructure relation. 
\end{Def}

\begin{Rem}
  $(\bkalpha,\subset)$ is not quite an abstract elementary class
  because $\phi_\alpha$ is not preserved by arbitrary unions of chains. In particular, parts (6) and (7) of the definition of 
$\phi_\alpha$ are not preserved by arbitrary unions. 
    However, this can be remedied by adding Skolem functions for
  parts (6) and (7). That is, for (6) introduce countably many new predicate symbols $(C^\beta_n)$, each $C^\beta_n$ of arity 
$n+3$, and require that $C^\beta_n(x,y,z,\vec{w})$ holds true if and only if $\vec{w}$ is the vector of all elements $w$ such 
that $T_\beta(x,y,w)$ is different than $T_\beta(x,z,w)$. By coherence there are only finitely many such $w$'s. For (7), 
introduce a new $4$-ary predicate symbol $P$ and require that $P(x,y,w,z)$ holds true if and only if $w<x$ and $z$ is such 
that $T_\beta(w,z,\bullet)\subset T_\beta(x,y,\bullet)$. 
Our results hold true even after such a change.
\end{Rem}

We prove that if $\alpha$ is finite, then $\bkalpha$ fails amalgamation in all cardinalities below 
$\aleph_\alpha$ (Lemma \ref{failap}), but amalgamation in $\aleph_\alpha$  holds trivially because all models of that size, if 
any,  are maximal (Lemma \ref{maximal}). By 
Theorem \ref{maintheorem}, for $\alpha\ge 2$ it is independent of ZFC+GCH
whether there are any models in $\bkalpha$ of size $\aleph_\alpha$.
We conclude that it is consistent with ZFC+GCH that
the amalgamation spectrum of $\bkalpha$ for $\alpha\geq 2$
is consistently empty and consistently equal
to $\{\aleph_\alpha\}$ (Theorem \ref{apempty}).

\begin{Lem}\label{endext}
  For all $\beta<\alpha<\omega_1$ and $A,B\in \bkalpha$,
  if $A\subset B$ and $\omega_\beta^A=\omega_\beta^B$,
  then $(\omega_{\beta+1}^B,<)$ end-extends $(\omega_{\beta+1}^A,<)$.
\end{Lem}
\begin{proof} Assume that there exists some $x\in \omega_{\beta+1}^B$ and $y\in\omega_{\beta+1}^A$ with 
$x<y$. By definition $L_\beta(y,\bullet)$ is a surjection from $\omega_\beta$ to $\downarrow y$. So, there exists some $z\in 
\omega_\beta^B$ such that $L^B_\beta(y,z)=x$. Since $\omega_\beta^A=\omega_\beta^B$, $z$ also belongs to $A$, which further
implies that $L^A_\beta(y,z)=x$.
So, $x$ must be an element of $\omega_{\beta+1}^A$.  
\end{proof}

\begin{Lem}\label{maximal} Assume $1\le \alpha<\omega$. All models, if any, in $\bkalpha$ of size $\aleph_\alpha$  are 
$\subset$-maximal.  
\end{Lem}
\begin{proof}
 First observe that a model $A$ of $\phi_\alpha$ that has size $\aleph_\alpha$ must satisfy $|\omega_\beta^A|=\aleph_\beta$, for 
all $\beta\le \alpha$, while any strict initial segment of $\omega_\beta^A$ must have size $<\aleph_\beta$. 

Next, assume that $A\subset B$ and both $A,B$ are of size $\aleph_\alpha$. We prove by induction on $\beta\le\alpha$, that 
$\omega_\beta^A=\omega_\beta^B$. 

For $\beta=0$, this follows by the fact that $(\omega_0,<)\cong(\omega,\in)$. 
For the inductive step, assume that for some $\beta<\alpha$, $\omega_\beta^A=\omega_\beta^B$. By Lemma \ref{endext}, 
$\omega_{\beta+1}^B$ is an end-extension of 
$\omega_{\beta+1}^A$. By the above observation, both $\omega_{\beta+1}^A$ and $\omega_{\beta+1}^B$ have size $\aleph_{\beta+1}$ 
and each strict initial segment has size $<\aleph_{\beta+1}$. This leads to a contradiction if we assume that 
$\omega_{\beta+1}^A$ is a strict initial segment of $\omega_{\beta+1}^B$. Thus, it must be the case that 
$\omega_{\beta+1}^A=\omega_{\beta+1}^B$. 

To finish the proof, observe that if $A\subset B$ and $A,B$ agree on all $\omega_\beta$'s, then $A,B$ are equal. 
\end{proof}

Next we prove a series of lemmas that lead to Lemma \ref{failap} where it is proved that $\bkalpha$ fails amalgamation below $\aleph_\alpha$.
These lemmas do not require $\alpha$ to be finite. 

\begin{Lem} \label{Lem:beta}
  Assume $1\le\beta<\alpha<\omega_1$ and $M\in\bkalpha$
  such that $\card{M}\in\{\aleph_0,\aleph_\beta\}$.
  Then there exists $N\in\bk_\beta$ of size $\card{M}$.
  If $\card{M}=\aleph_\beta$, then $N$ can be chosen
  to be $\subset$-maximal too.
\end{Lem}
\begin{proof}
  If $\card{M}=\aleph_0$, then let $J=\beta+1$. Otherwise,
  let $J$ denote the set of all $\gamma\leq\alpha$ with the property that
  $\card{\omega_\delta^M}<\card{\omega_\gamma^M}$ for all $\delta<\gamma$.
  In both cases, there is a unique order isomorphism $g\colon\beta+1\to J$.
  Note that $g$ is a continuous map from $\beta+1$ to $\alpha+1$
  and $g$ maps successor ordinals to successor ordinals.
  For each $\gamma<\beta$, choose a bijection
  $f_\gamma$ from $\omega_{g(\gamma)}^M$ to $\omega_{g(\gamma+1)-1}^M$.
  We construct $N\in\bk_{\beta}$
  with universe $\omega_{g(\beta)}^M$ by
  relabeling the $\gamma^{\text{th}}$ pseudotree of $M$
  for $\gamma\in \alpha\cap J$
  and eliminating the $\gamma^{\text{th}}$ pseudotree of $M$
  for $\gamma\in\alpha\setminus J$.
\begin{enumerate}
  \item $\omega_\gamma^N=\omega_{g(\gamma)}^M$ for all $\gamma\leq\beta$.
  \item For each $\gamma<\beta$,
    $x\in\omega_{\gamma+1}^N$, and $y\in\omega_{\gamma}^N$:
    \begin{enumerate}
    \item Let $L^{N}_\gamma(x,y)=L^{M}_{g(\gamma+1)-1}(x,f_\gamma(y))$.
      This defines a surjection from $\omega_{\gamma}^N$ to $\downarrow x$. 
    \item Let $T^{N}_\gamma(x,y,\bullet)=T^{M}_{g(\gamma+1)-1}(x,f_\gamma(y),\bullet)$.
    \item Let $S^{N}_\gamma(x,y)=f_\gamma^{-1}(S^{M}_{g(\gamma+1)-1}(x,f_\gamma(y)))$.
    \end{enumerate}
\end{enumerate}
In the case $\card{M}=\aleph_\beta$,
to see that $N$ is $\subset$-maximal, note that
$\card{\omega_\gamma^N}=\card{\omega_{g(\gamma)}^M}=\aleph_\gamma$
for each $\gamma\leq\beta$, which in turn implies that
each strict initial segment of $\omega_\gamma^N$ has size $<\aleph_\gamma$,
for each $\gamma\leq\beta$.
Therefore, the proof of Lemma~\ref{maximal}
shows that $N$ is $\subset$-maximal.
\end{proof}

\begin{Def}
  Given $\alpha<\alpha'$, $M\in\bk_\alpha$, and $M'\in\bk_{\alpha'}$,
  we say that $M'$ \emph{end-extends} $M$ if
  $M$ and $M'$ agree on $\omega_{\gamma}$ for all $\gamma\leq\alpha$
  and  on $L_{\gamma}, T_{\gamma}, S_{\gamma}$ for all $\gamma<\alpha$. 
\end{Def}

\begin{Lem}\label{Lem:betaplus} Assume $1\le\alpha<\omega_1$,
  $M\in\bk_\alpha$, and let $L$ be a linear order of size $\leq\card{M}$.
  Then there is $N\in\bk_{\alpha+1}$ of size $\card{M}$
  such that $N$ end-extends $M$ and
  $\omega_{\alpha+1}^N\setminus \omega_\alpha^N\cong L$.
  
  Moreover, we may choose $N$ such that,
  for each linear order $L'$ of size $|M|$ that end-extends $L$,
  there exists $N'\in\bkalpha$ such that $N\subset N'$,
  $\card{N}=\card{N'}$, and $\omega_{\alpha+1}^{N'}\setminus\omega_\alpha^{N'}\cong L'$.
\end{Lem}
\begin{proof}
  
  Given $L$ and $L'$, we will construct $N$ and $N'$ concurrently 
  so that $N$ depends on $L$ but not on $L'$.
  Without loss of generality we may assume that
  $L'$ and $\omega_\alpha^M$ are disjoint.

  Fix some injection $g$ from $\omega^M_\alpha\cup  L$ to $\omega^M_\alpha$ so that both 
  the range of $g$ and its complement have size $|M|$. Extend $g$ to an injection $f$ from 
  $\omega^M_\alpha\cup  L'$ to $\omega^M_\alpha$. Let $0^M$ denote $\min(\omega_0^M)$.
  End-extend $M$ to $N'\in\bk_{\alpha+1}$ as follows.
  \begin{enumerate}
  \item $\omega^{N'}_{\alpha+1}=\omega^M_\alpha+L'$.
  \item For $0^M\not=x\in \omega^{N'}_{\alpha+1}$, let $L^{N'}_\alpha(x,\bullet)$
    be an arbitrary surjection from $\omega_\alpha^M$ to $\downarrow x$.
  \item For $x\in \omega^{N'}_{\alpha+1}$ and $y\in \omega^{N'}_{\alpha}$,
    let $S^{N'}_\alpha(x,y)=f(x)$ and
    \(T^{N'}_{\alpha}(x,y,z)= 0^M$ for all $z<x$.
  \end{enumerate}
  If we also stipulate that $\omega^{N}_{\alpha+1}=\omega^M_\alpha+L$, the above implicitly defines $L_{\alpha}^N$, 
$S_{\alpha}^N$, and $T_{\alpha}^N$
  so that they depend on $L$ but not on $L'$.
  
  \end{proof}

\begin{Cor}\label{Cor:alpha}
  Assume $1\le\alpha<\alpha'<\omega_1$, $M\in\bkalpha$,
  and a sequence of linear orders
  $L_\gamma$ for $\alpha\leq\gamma<\alpha'$ each such that
  $\card{L_\gamma}\leq\card{M}$.
  Then there is $N\in\bk_{\alpha'}$ of size $\card{M}$
  that end-extends $M$ and satisfies
  $\omega^{N}_{\gamma+1}\setminus \omega^{N}_\gamma\cong L_\gamma$
  for $\alpha\leq\gamma<\alpha'$.
  
  Moreover, we may choose $N$ such that,
  if $M'\in \bkalpha$ satisfies $M\subset M'$ and $|M|=|M'|$, then there exists some end-extension $N'\in \bk_{\alpha'}$ of $M'$ 
 that satisfies $|N'|=|M'|$, $N\subset N'$ and  $\omega_{\gamma+1}^{N'}\setminus \omega^{N'}_\gamma\cong L_\gamma$
  for $\alpha\leq\gamma<\alpha'$.
   \end{Cor}
\begin{proof}
  Create $N$ by repeatedly end-extending $M$
  using Lemma~\ref{Lem:betaplus} at successor stages
  and unions at limit stages.
  
  To prove the claim about $N'$ we follow the proof of Lemma \ref{Lem:betaplus}. The differences are now that  (a) 
$\omega_\alpha^N$ and $\omega_\alpha^{N'}$ may not be the same and (b) the construction of $N$ and $N'$ 
guarantees that $\omega_{\alpha'}^{N'}\setminus\omega_{\alpha'}^{N}$ equals $\omega_{\alpha}^{M'}\setminus\omega_{\alpha}^{M}$, 
i.e. no new points are added to $N$ other than the points added to $M$.

The proof is by induction on $\gamma$. The limit stages are trivial, so we describe only how $L_{\gamma},S_{\gamma}$ and 
$T_{\gamma}$ are defined on the successor stages.
\begin{itemize}
 \item For $0^M\not=x\in \omega^{N'}_{\gamma+1}$, let $L^{N'}_\gamma(x,\bullet)$ equal $L^{N}_\gamma(x,\bullet)$ 
when restricted to domain $\omega_{\gamma}^N$, and let $L^{N'}_\gamma(x,\bullet)$ be the identity otherwise. 
    \item Similarly, for $x\in \omega^{N'}_{\gamma+1}$, let $S^{N'}_\gamma(x,\bullet)$  equal $S^{N}_\gamma(x,\bullet)$  
when restricted to domain $\omega_{\gamma}^N$, and let $S^{N'}_\gamma(x,\bullet)$  be the identity otherwise. 
     \item For $x\in \omega^{N'}_{\gamma+1}$ and $y\in \omega^{N'}_{\gamma}$,
    let \(T^{N'}_{\gamma}(x,y,z)= 0^M$ for all $z<x$.
\end{itemize}
\end{proof}

\begin{Lem}\label{failap} Let $1\le\beta<\alpha<\omega_1$ and $\gamma\in\{0,\beta\}$
  and assume that $\bkalpha$ has a model of size $\aleph_\gamma$.
  Then amalgamation fails in $\aleph_\gamma$. 
\end{Lem}
\begin{proof} We give an example of a triple $(A,B,C)$ in 
$\bkalpha$ that can not be amalgamated. The reason that amalgamation fails is that linear orders fail amalgamation under 
end-extension. 

Assume $M\in\bkalpha$ and $|M|=\aleph_\gamma$. The proof splits into two cases: $\gamma=\beta>0$ and $\gamma=0$. We give the 
details for the first case and sketch the proof of the second case.

By Lemma \ref{Lem:beta}, there exists a $\subset$-maximal $N\in 
\bk_\gamma$ with $|N|=\aleph_\gamma$. End-extend $N$ using 
Lemma \ref{Lem:betaplus} to create three models $A', B', C' \in \bk_{\gamma+1}$ that satisfy 
the following:
\begin{enumerate}
 \item $A'\subset B'$ and $A'\subset C'$
 \item $|A'|=|B'|=|C'|=\aleph_\gamma$
 \item $\omega_{\gamma}^{A'}=\omega_{\gamma}^{B'}=\omega_{\gamma}^{C'}=\omega_{\gamma}^N$
 \item $\omega_{\gamma+1}^{A'}=\omega^{A'}_{\gamma}+\omega$
 \item $\omega_{\gamma+1}^{B'}=\omega^{A'}_{\gamma}+\omega\cdot 2$
 \item $\omega_{\gamma+1}^{C'}=\omega^{A'}_{\gamma}+\omega+\mathbb{Q}$
 \end{enumerate}
 Then use Corollary \ref{Cor:alpha} to end-extend $A', B',C'$ to some $A,B,C\in\bkalpha$ such that $A\subset B$ and $A\subset 
 C$. Assume $D$ is an amalgam of $B$ and $C$ over $A$.
 Then $\omega_\gamma^D=\omega^N_\gamma$ by maximality of $N$.
 By Lemma \ref{endext},
$\omega_{\gamma+1}^D$ must be an end-extension
of both $\omega_{\gamma+1}^B$ and $\omega_{\gamma+1}^C$.
But this is impossible.  

For the case when $\gamma=0$ construct three models $A,B,C\in \bkalpha$ with 
$\omega_{1}^A=\omega\cdot 2$, $\omega_{1}^{B}=\omega\cdot 3$ and $\omega_{1}^{C}=\omega\cdot 2+\mathbb{Q}$. The same argument 
proves that they can not be amalgamated in $\bkalpha$.
\end{proof}

\begin{Thm}\label{apempty}

Assume $1\le\alpha<\omega$. 
The amalgamation spectrum of $\bkalpha$ is equal to $\{\aleph_\alpha\}$, if there are 
models in $\bkalpha$ of size $\aleph_\alpha$. Otherwise it is empty. 
\end{Thm}
\begin{proof} First recall that by \ref{mainlemma}(3), $\bkalpha$ has no models of size greater than $\aleph_\alpha$. By Lemma 
\ref{failap} amalgamation fails for all cardinals below $\aleph_\alpha$. If 
there are models in $\bkalpha$ of size $\aleph_\alpha$, then $\aleph_\alpha$-amalgamation holds 
trivially by Lemma \ref{maximal}. In this case the amalgamation spectrum is equal to 
$\{\aleph_\alpha\}$. Otherwise, the amalgamation spectrum  is empty. 
 \end{proof}

\begin{Cor}\label{apone}
  The amalgamation spectrum of $\bk_1$ is $\{\aleph_1\}$.
\end{Cor}
\begin{proof} The existence of \csat{\aleph_1} follows from Theorem \ref{Konig}, because $\square_{\aleph_0}$ holds 
trivially. 
\end{proof}

\begin{Thm}\label{nonabsoluteness} The following statements are \emph{not absolute} for transitive models of ZFC.
\begin{enumerate}[(a)]
 \item The amalgamation spectrum of an $\lomegaone$-sentence is empty. 
\item For finite $n\ge 2$ and $\phi$ an $\lomegaone$-sentence, $\aleph_n$ belongs to the amalgamation spectrum of $\phi$.
\end{enumerate}
The results remain true even if we consider transitive 
models of ZFC+GCH.
\end{Thm}
\begin{proof} By Theorem \ref{apempty} and Theorem \ref{maintheorem}.
\end{proof}

A couple of notes: Theorem \ref{nonabsoluteness} covers all cardinals $\aleph_n$, with $n$ finite and $n\ge 2$. For $n=0$, 
$\aleph_0$-amalgamation is absolute by an easy application of Shoenfield's absoluteness. The question for $n=1$ remains open. 

Lastly we prove that our examples can not be used to resolve the absoluteness question of $\aleph_\alpha$-amalgamation, for 
$\omega\le\alpha<\omega_1$,  under GCH. The reason is that in this case $\bkalpha$ has empty amalgamation spectrum. 

\begin{Lem}\label{treesurgery}
  Assume $\omega\le\alpha<\omega_1$ and $M\in\bkalpha$.
  Let $K$ be a countable linear order.
  Then there is a model $R$ in  $\bkalpha$ of size $\card{M}$
  such that $\omega_1^R\setminus\omega_0^R\cong K$.

  Moreover, we may choose $R$ such that,
  for each countable linear order $J$ that end-extends $K$,
  there exists $N\in\bkalpha$ such that $R\subset N$,
  $\card{N}=\card{R}$, and $\omega_1^N\setminus\omega_0^N\cong J$.
\end{Lem}
\begin{proof}
  Given $K$ and $J$, we will construct $R$ and $N$ in parallel,
  taking care that $R$ depends on $K$ but not on $J$.
  The idea is that we move all the pseudotrees of $M$ one level higher
  and introduce a new pseudotree at the bottom.

  For each $\beta\le\alpha$, define \(\sigma(\beta)\) to be
  $\beta-1$ if $0<\beta<\omega$ and $\beta$ otherwise.
  Without loss of generality, assume $(\omega_0^M,<^M)=(\omega,\in)$;
  then define $\omega_0^N=\omega$ and
  $\omega_\beta^N=\omega+J+(\omega_{\sigma(\beta)}^M\setminus\omega)$
  for all $\beta>0$. In particular, $\omega_1^N=\omega+J$ and
  $\omega_2^N=\omega+J+(\omega_1^M\setminus\omega)$.
  For all $\beta\le\alpha$,
  let $\omega_\beta^R=\omega_\beta^N\setminus(J\setminus K)$.

  Next, we define the bottom pseudotree of $N$
  and, implicitly, the bottom pseudotree of $R$.
  Choose an injection $g\colon\omega+K\to\omega$ with
  co-infinite range and then extend it to an injection
  $f\colon\omega+J\to\omega$.
  For each $x\in\omega+J$ and $y\in\omega$, declare that:
  \begin{itemize}
  \item $L_0^N(x,\bullet)$ is an arbitrary surjection
    from $\omega$ to $\downarrow^N x$ if $x\not=0$.
  \item $T_0^N(x,y,z)=0$ for all $z<^N x$.
  \item $S_0^N(x,y)=f(x)$.
  \end{itemize}
  The above implicitly defines $L_0^R$, $S_0^R$, and $T_0^R$
  so that they depend on $K$ but not on $J$.

  Given $1\le\beta<\alpha$,
  observe that $\omega_\beta^N=\omega_{\sigma(\beta)}^M\cup J$ and
  $\omega_{\beta+1}^N=\omega_{\sigma(\beta)+1}^M\cup J$.
  Then declare the following for each
  $x\in\omega_{\beta+1}^N\setminus\{0\}$ and $y\in\omega_\beta^N$.
  \[L_\beta^N(x,y)=\begin{cases}
  L_{\sigma(\beta)}^M(x,y)&:x,y\not\in J\\
  y&:x>^N y\in J\\
  0&:x\le^N y\in J
  \end{cases}\]
  The above defines our needed surjections $L^N_\beta(x,\bullet)$
  and $L^R_\beta(x,\bullet)$, except in the case where $x\in J$ and $y\in \omega^M_{\sigma(\beta)}$.
  To define $L^N_\beta$ in this case, fix some surjection $f_\beta$ from $\omega^M_{\sigma(\beta)}$ to $\omega$ and let 
$L^N_\beta(x,y)=f_\beta(y)$. 

This completes the definition of $L^N_\beta(x,\bullet)$
  and $L^R_\beta(x,\bullet)$, and notice that the latter
  depends on $K$ but not on $J$.

  Next, we begin defining the $\beta^{\text{th}}$ pseudotree of $N$
  by splicing constant functions into
  the $\sigma(\beta)^{\text{th}}$ pseudotree of $M$.
  For each $x\in\omega_{\sigma(\beta)+1}^M$ and $y\in\omega_{\sigma(\beta)}^M$,
  let \[T_\beta^N(x,y,z)=
  \begin{cases}
    T_{\sigma(\beta)}^M(x,y,z)&:x>^N z\not\in J\\
    0&:x>^N z\in J
  \end{cases}.\]
  The parts of the pseudotrees $T_\beta^N$ and $T_\beta^R$
  defined so far are coherent.
  We now fill in the missing levels indexed by $J$
  so as to also achieve downward closure.
  Fix $x_0\in\omega_1^M\setminus\omega$ and
  let $(g_\beta,h_\beta)$ map $\omega_{\sigma(\beta)}^M$ onto the set of pairs
  \[\left\{(x,y)\suchthat x_0\ge^M x\not\in\omega
  \text{ and }y\in\omega_{\sigma(\beta)}^M\right\}.\]
  Then declare that, for each $x\in J$ and $y\in\omega_{\sigma(\beta)}^M$,
  \[T_\beta^N(x,y,z)=
  \begin{cases}
    T_{\sigma(\beta)}^M(g_\beta(y),h_\beta(y),z)&:z\in\omega\\
    0&:x>^N z\in J
  \end{cases}.\]
  Observe that so far $T_\beta^R$ depends on $K$ but not on $J$.

  To witness specialness, declare that, for each
  $x\in\omega_{\beta+1}^N$ and $y\in\omega_{\sigma(\beta)}^M$,
  \[S_\beta^N(x,y)=
  \begin{cases}
    S_{\sigma(\beta)}^M(x,y)&:x\not\in J\\
    x&:x\in J
  \end{cases}.\]
  Finally, for each $x\in\omega_{\sigma(\beta)+1}^M$ and $y\in J$,
  let $T_\beta^N(x,y,\bullet)=T_\beta^N(x,0,\bullet)$
  and $S_\beta^N(x,y)=S_\beta^N(x,0)$.
  This completes the construction of $N$ and $R$.
  We have also implicitly defined $T_\beta^R$ and $S_\beta^R$
  so that they depend on $K$ but not on $J$.
\end{proof}

\begin{Thm} Assume $\omega\le\alpha<\omega_1$. 
The amalgamation spectrum of $\bkalpha$ is empty.
\end{Thm}
\begin{proof}
 By Lemma \ref{failap}, $\bkalpha$ fails amalgamation in $\aleph_\beta$, for all $\beta<\alpha$. We prove that this is the case 
even for $\beta=\alpha$. If $\bkalpha$ has no models of size $\aleph_\alpha$, the result is trivial. So, assume that is a model 
$M$ of size $\aleph_\alpha$.

We use the same method as for Lemma \ref{failap}. In particular, we construct a triple $(A,B,C)$ with the following properties:
\begin{enumerate}
 \item $A\subset B$ and $A\subset C$
 \item $|A|=|B|=|C|=\aleph_\alpha$
 \item $\omega_{1}^A=\omega\cdot 2$
 \item $\omega_{1}^B=\omega\cdot 3$
 \item $\omega_{1}^C=\omega\cdot 2+\mathbb{Q}$
\end{enumerate}
This is possible by applying Lemma \ref{treesurgery} twice; once for the pair $\omega\cdot 2\subset\omega\cdot 3$ and a second time for the pair 
$\omega\cdot 2\subset\omega\cdot 2+\mathbb{Q}$. 

As in Lemma \ref{failap}, if $D$ were an amalgam of $B$ and $C$ over $A$, then $\omega_1^D$ would must be an end-extension 
of both $\omega_{1}^B$ and $\omega_{1}^C$, which is impossible.  
\end{proof}

\section{Open Problems}
 The following are some questions that remain open. Some of the questions do not bear much resemblance to the results of this 
paper. Nevertheless we encountered these questions during our search for a proof to Theorem \ref{maintheorem}.
 \begin{enumerate}
  \item  Can the consistency strength of our non-absoluteness theorem  be further reduced? In particular, is it  
possible to prove the same result without any large cardinal assumptions? 
\item If $\alpha$ is countably infinite, is $\aleph_\alpha$-amalgamation non-absolute for transitive models of ZFC+GCH?
\item Is $\aleph_1$-amalgamation for $\lomegaone$-sentences absolute for  transitive models of ZFC? 
\item The way we proved non-absoluteness of amalgamation in $\aleph_n$, for finite $n$, is by choosing appropriate set-theoretic 
assumptions that affect the model-existence spectrum. If there are no models in $\aleph_n$ then the amalgamation question 
becomes void. Can we prove non-absoluteness of amalgamation in $\aleph_n$ in the following stronger form: There are two 
transitive 
models of ZFC, say $V\subset W$, with the same ordinals, a 
sentence $\phi$ that belongs to $(\lomegaone)^V$, both $V$ and $W$ satisfy ``$\phi$ has a model of size $\aleph_n$'', and 
$V,W$ disagree on ``models of $\phi$ of size $\aleph_n$ satisfy amalgamation''? Same question is open for 
$\aleph_n$-joint 
embedding.
\item The principle $\square_\kappa$ asserts the existence of a square sequence, i.e. a sequence $<C_\alpha|\alpha\in 
Lim(\kappa^+)>$ that satisfies (i) $C_\alpha$ is a club of $\alpha$, (ii) if $cf(\alpha)<\kappa$, then $|C_a|<\kappa$ and 
 (iii) if $\beta\in Lim(C_\alpha)$, then $C_\beta=C_\alpha\cap \beta$. Are there any $\kappa^+$-like linear orders 
$(L,<)$ (other than well-orders) for which the existence of a sequence $<C_\alpha|\alpha\in Lim(L)>$ that satisfies (i)-(iii) is 
independent of ZFC?
\item The proof of Lemma~\ref{mainlemma} does not quite recover
  a \csat{\aleph_2}\ from an $\aleph_\alpha$-sized model of $\phi_\alpha$.
  It merely recovers a \sat{\aleph_2} that embeds in a coherent pseudotree.
  Is there an $\lomegaone$-sentence for which
  existence of a model of size $\aleph_2$
  entails a \csat{\aleph_2}?
\item One strategy for reducing our large cardinal assumption
  from Mahlo to inaccessible is to attempt to code Kurepa trees
  using formulas satisfied by higher-gap simplified morasses.
  The following test question captures
  the core obstacle to this strategy.
  Assume $V=L$. Choose $\fM=(L(\delta),\in)$ such that
  $\omega_4<\delta<\omega_5$ and $\fM\prec(L(\omega_5),\in)$.
  Let $G$ be a generic filter of the Miller-like version of Namba forcing.
  (This forcing is the Nm defined in XI.4.1 of~\cite{Shelah};
  the Laver-like version of Namba forcing
  is the $\mathrm{Nm}'$ defined in the remark after XI.4.1A.)
  Then, in $V[G]$, GCH and the regularity of $\omega_1^V$ and $\omega_4^V$
  are preserved but $\cf(\omega_2^V)$ collapses to $\omega$
  and $\cf(\omega_3^V)$ collapses to $\omega_1$.
  In $V$, let $\psi$ be an $\lomegaone$ formula that defines
  a binary relation on the structure $\fM$,
  possibly using parameters from $\fM$.
  Can $\psi$ be chosen independently of $G$ such that in $V[G]$
  we have $\card{\dom\left(\psi^{\fM}\right)}=\aleph_2$,
  $\card{\ran\left(\psi^{\fM}\right)}=\aleph_1$, and
  $\left\{\psi^{\fM}[x]\cap y\suchthat x\in\dom\left(\psi^{\fM}\right)\right\}$
  countable for all countable sets $y$?
\item In the above test question, we specified ``Miller-like''
  because we can prove, assuming CH, that this version of Namba forcing
  does not add cofinal branches to $\omega_1$-trees in the ground model,
  thus opening the door to rcs iterated forcing extensions without
  any Kurepa trees (assuming an inaccessible in the ground model).
  However, our proof's fusion argument works for Miller and Sacks
  but not Laver type tree forcings.
  This leads us to ask, is it consistent with CH that Laver forcing
  adds a cofinal branch to some $\omega_1$-tree in the ground model?
\end{enumerate}

\end{document}